\documentclass[11pt,twoside]{article}
\usepackage{amsfonts}
\usepackage{amsmath}
\usepackage{amssymb}
\usepackage{multicol}
\usepackage{graphics}
\usepackage{stmaryrd}
\usepackage{cite}
\newtheorem{theorem}{Theorem}[section]
\newtheorem{lemma}[theorem]{Lemma}

\newtheorem{definition}[theorem]{Definition}
\newtheorem{remark}[theorem]{Remark}

\numberwithin{equation}{section}
\newenvironment{proof}[1][Proof]{\noindent\textbf{#1.} }{\hfill $\Box$}
\allowdisplaybreaks

 \makeatletter
\begin{document}
\author{Shuxia Pan$^1$\thanks{Corresponding author. E-mail: shxpan@yeah.net.} \thanks{Supported by NSF of Gansu
Province of China (1208RJYA004) and the Development Program for Outstanding Young Teachers in Lanzhou University of Technology (1010ZCX019).} and Guo Lin$^2$ \thanks{Supported by NSF of China (11101194).} \\
$^1${Department of Applied Mathematics, Lanzhou University of Technology,}\\
{Lanzhou, Gansu 730050, People's Republic of China}\\
$^2${School of Mathematics and Statistics, Lanzhou University,}\\
{Lanzhou, Gansu 730000, People's Republic of China}}

\title{\textbf{Coinvasion-Coexistence Traveling Wave Solutions of an Integro-Difference Competition System}}\maketitle

\begin{abstract}
This paper is concerned with the traveling wave solutions of an integro-difference competition system, of which the purpose is to model the coinvasion-coexistence process of two competitors with age structure.  The existence of nontrivial traveling wave solutions is obtained by constructing generalized upper and lower solutions.  The asymptotic and nonexistence of traveling wave solutions are proved by combining the theory of asymptotic spreading with the idea of contracting rectangle.


\end{abstract}

\begin{center}
Submitted to \emph{Journal of Difference Equations and Applications.}
\end{center}

\newpage

\section{Introduction}
\noindent

Competition is a universal phenomenon in ecological communities due to the limitation of resources. In population dynamics, there are many important evolutionary competitive systems formulating different competitive mechanism, such as the Lotka-Volterra system, Gilpin-Ayala competition model, the following difference system (see Cushing et al. \cite{cush})
\begin{equation}\label{cus}
\begin{cases}
p_{n+1}=
\frac{(1+r_1)p_n}{1+r_1(p_n+a_1q_n)},\\
q_{n+1}=
\frac{(1+r_2)q_n}{1+r_2(q_n+a_2p_n)},
\end{cases}
\end{equation}
and the following one (see Hassell and Comins \cite{haco}, Kang and Smith \cite{kang}, Li et al. \cite{lijia})
\begin{equation}\label{ricker}
\begin{cases}
X_{n+1}=
X_{n}e^{r_{1}(1-X_{n}-a_{1}Y_{n})},\\
Y_{n+1}=
Y_{n}e^{r_{2}(1-Y_{n}-a_{2}X_{n})},
\end{cases}
\end{equation}
in which  $n\in \mathbb{N}\bigcup \{0\},$ and $r_1>0,r_2>0, a_1\ge 0, a_2\ge 0$ are four constants.
From the viewpoint of population dynamics,  \eqref{ricker} implies that all the interspecific and intraspecific competition is confined to one of the developmental stages \cite[Model 9]{haco}. However, for some populations, age structure can influence population size and growth in a major way \cite[Section 1.7]{murray}. To describe the age structure in population dynamics, one recipe is the difference equations of higher order \cite[Section 2.5]{murray}. In particular, when both the interspecific competition and intraspecific competition can occur among the individuals with the same and the different age, and the competition mechanism is similar to that among the individuals with the same age, we can modify \eqref{ricker} as follows to reflect  the phenomenon
\begin{equation}\label{0001}
\begin{cases}
X_{n+1}=X_{n}e^{r_{1}\left(
1-X_{n}-\sum_{i=1}^{m}a_{i}X_{n-i}-\sum_{i=0}^{m}b_{i}Y_{n-i}\right) },\\
Y_{n+1}=Y_{n}e^{r_{2}\left(
1-Y_{n}-\sum_{i=1}^{m}e_{i}Y_{n-i}-\sum_{i=0}^{m}f_{i}X_{n-i}\right) },
\end{cases}
\end{equation}
in which $m\in \mathbb{N} \bigcup \{0\}$ is a constant, $a_i\ge 0, b_i\ge 0, e_i\ge 0, f_i\ge 0$ are constants describing the interspecific and intraspecific competition.

Although the spatially homogeneous evolutionary systems including \eqref{cus}-\eqref{0001} play very important roles in illustrating many processes such as the oscillatory levels of certain fish catches in the Adriatic \cite{v},  pharmacodynamics of HAART \cite{hod} and chaos \cite[Section 2.3]{murray}, it is inevitable to involve the spatial distribution of individuals with the ability of random walk in the problems such as the biology invasion, central pattern generator \cite{murray}.
In particular, the spatial propagation of evolutionary systems has been widely studied  since Fisher \cite{fisher}, and traveling wave solution is a useful index formulating the propagation \cite{vol}. In 1982, Weinberger \cite{wein} derived an evolutionary equation with discrete temporal variable and studied its traveling wave solutions. Since then, much attention has been paid to the spatial propagation of the corresponding spatio-temporal models of difference equations including \eqref{cus} by combining the comparison principle with other techniques, we refer to Lewis et al. \cite{lewis}, Weinberger et al. \cite{weinberger} for the traveling wave solutions reflecting competition-exclusion process, and Li \cite{libingtuan}, Lin and Li \cite{linlidc} and Lin et al. \cite{llrjmb} for the traveling wave solutions modeling competition-coinvasion process. It should be noted that from the viewpoint of monotone dynamical systems,  \eqref{cus} with any $r_1,r_2\in (0,\infty)$ admits proper comparison principle in an invariant region (e.g., $[0,1]\times [0,1]$).

Recently, Wang and Castillo-Chavez \cite{wc} investigated the spatial propagation of the following integro-difference system
\begin{equation}\label{1.10}
\begin{cases}
X_{n+1}(x)=\int_{\mathbb{R}}
X_{n}(y)e^{r_{1}(1-X_{n}(y)-a_{1}Y_{n}(y))}k_{1}(x-y)dy,\\
Y_{n+1}(x)=\int_{\mathbb{R}}
Y_{n}(y)e^{r_{2}(1-Y_{n}(y)-a_{2}X_{n}(y))}k_{2}(x-y)dy,
\end{cases}
\end{equation}
where $ n+1\in \mathbb{N}, x\in\mathbb{R},$  $X_n(x), Y_n(x)$ denote the densities
of two competitors at time $n$
at location $x$ in population dynamics, respectively, and the kernels $k_i, i=1,2,$ are probability functions describing the spatial dispersal of individuals. Li and Li \cite{lix} also studied the asymptotic behavior of traveling wave solutions of \eqref{1.10}.

To understand the difficulty in the study of \eqref{1.10}, we first give some properties of  \eqref{ricker}. When  $r_1>1$ and $r_2 >1$ hold, it is clear that
\[
\left[ 0,\frac{e^{r_{1}-1}}{r_{1}}\right] \times \left[ 0,\frac{e^{r_{2}-1}}{%
r_{2}}\right]
\]
is an invariant region of \eqref{ricker}, but
\[
xe^{r_{1}(1-x-a_{1}y)},\text{ \ \ \ }ye^{r_{2}(1-y-a_{2}x)}
\]
are not monotone for
\[
x\in \left[ 0,\frac{e^{r_{1}-1}}{r_{1}}\right] , \,\,\, y\in \left[ 0,\frac{e^{r_{2}-1}}{%
r_{2}}\right]
\]
such that we cannot find desired comparison principle of \eqref{ricker} similar to that of \eqref{cus}. Because of the invalidation of comparison principle, the study of \eqref{1.10} will be harder than that of the corresponding integro-difference system of \eqref{cus} from the viewpoint of monotone dynamical systems. When the competition-exclusion process of \eqref{1} is concerned, it is a locally cooperative system after a change of variables. In Wang and Castillo-Chavez \cite{wc}, the authors studied the traveling wave solutions and asymptotic spreading
by constructing two auxiliary cooperative systems generating monotone semiflows (we refer to \cite{hsuzhao,llw,yi} for constructing auxiliary monotone equations in the study of scalar nonmonotone integro-difference equations).

Very recently, Li and Li \cite{lili} investigated the existence of nontrivial positive traveling wave solutions of \eqref{1.10} if
\begin{description}
\item[$\bullet $] $k_1,k_2$ take the form of Gaussian kernels;
\item[$\bullet $] $a_1,a_2 \in [0,1);$
\item[$\bullet$] $r_1,r_2 \in (0,1].$
\end{description}
With these assumptions and wave speed larger than a threshold, the authors established the existence of traveling wave solutions connecting the trivial equilibrium with the positive one by constructing upper and lower solutions.
Clearly, \eqref{1.10} with these assumptions satisfies comparison principle and is a special form of
\begin{equation}\label{1}
\begin{cases}
X_{n+1}(x)=\int_{\mathbb{R}}X_{n}(y)e^{r_{1}\left(
1-X_{n}(y)-\sum_{i=1}^{m}a_{i}X_{n-i}(y)-\sum_{i=0}^{m}b_{i}Y_{n-i}(y)\right) }k_{1}(x-y)dy,\\
Y_{n+1}(x)=\int_{%
\mathbb{R}}Y_{n}(y)e^{r_{2}\left(
1-Y_{n}(y)-\sum_{i=1}^{m}e_{i}Y_{n-i}(y)-\sum_{i=0}^{m}f_{i}X_{n-i}(y)\right) }k_{2}(x-y)dy,
\end{cases}
\end{equation}
in which  $n+1\in\mathbb{N}, x\in\mathbb{R}.$   For the parameters and kernel functions in \eqref{1}, we first make the following assumptions:
\begin{description}
\item[(A1)] $k_i$ is Lebesgue measurable and integrable
 on $\mathbb{R}$ and
$\int_{\mathbb{R}}k_i(y)dy=1,i=1,2;$
\item[(A2)] $k_i(y)=k_i(-y)\ge 0,y\in \mathbb{R}, $ and  for each $\lambda \in\mathbb{R},\int_{\mathbb{R}}k_i(y)e^{\lambda y}dy<\infty , i=1,2;$
\item[(A3)] $r_1>0,r_2>0,a_i\ge 0,e_i\ge 0$ for $i\in \{1,2,\cdots, m\};$
\item[(A4)] $b_i\ge 0, f_i\ge 0$ for $i\in \{0, 1,2,\cdots, m\};$
\item[(A5)] $r_1,r_2\in (0,1]$ and
\[
\sum_{i=1}^{m}a_{i}+\sum_{i=0}^{m}b_{i}<1, \sum_{i=1}^{m}e_{i}+\sum_{i=0}^{m}f_{i}<1.
\]
\end{description}
Similar to the study of \eqref{1.10} in \cite{lili}, we shall investigate the spatial propagation of \eqref{1} by traveling wave solutions  that formulate the synchronous invasion of two species admitting age structure,
and we refer to \cite{shi} for the historical records of several competitive species which successfully invaded a habitat together.

To consider the traveling wave solutions of \eqref{1}, the first difficulty is that the comparison principle may fail such that it is not an easy job to define upper and lower solutions similar to those in Lin et al. \cite{llrjmb}. Moreover, when $\sum_{i=0}^{m}b_{i}=0$ and $\sum_{i=0}^{m}f_{i}=0,$ the definition of upper and lower solutions in Lin and Li \cite{ll-jmaa} and  the techniques in Pan and Li \cite{pl} can not be applied to \eqref{1} if $\sum_{i=1}^ma_i >0$ or $\sum_{i=1}^me_i >0$ is true.

Very recently, Lin \cite{lin} developed the theory of traveling wave solutions of integro-difference systems of higher order. Applying the generalized upper and lower solutions, the existence of traveling wave solutions was obtained. In particular, \cite{lin} does not require the asymptotic behavior of generalized upper and lower solutions even if the system is not cooperative. Further applying the contracting rectangle, the authors presented a sufficient condition of the asymptotic behavior of traveling wave solutions.
By the theory in \cite{lin}, we prove the existence of nontrivial positive traveling wave solutions of \eqref{1} if (A1)-(A4) hold. When (A1)-(A5) are true, we further study the asymptotic behavior of traveling wave solutions, which includes/improves the main results of Li and Li \cite{lili}. Under the assumptions (A1)-(A4), we also obtain the existence of nontrivial traveling wave solutions formulating successful invasion of two competitors.  Furthermore, we investigate the nonexistence of traveling wave solutions by the theory of asymptotic spreading of integro-difference equations, which remains true for the model in \cite{lili}.

The rest of this paper is organized as follows. In Section 2, we list some necessary preliminaries. The existence of nontrivial positive traveling wave solutions is proved in Section 3. To answer the asymptotic behavior of traveling wave solutions in Section 4, we apply the contracting rectangle. Finally, the nonexistence of traveling wave solutions is confirmed in Section 5.

\section{Preliminaries}
\noindent

Let
$X$ be the set of uniformly continuous
and bounded functions from $\mathbb{R}$ to
$\mathbb{R}^2$. Moreover,  we shall use the standard partial order in $\mathbb{R}^2$ or $X.$
Let $\|\cdot \|$ be the supremum norm in $\mathbb{R}^2$ and $\mu >0$ be a constant, we  define
\[
B_{\mu }\left( \mathbb{R},\mathbb{R}^{2}\right) =\left\{\Phi \in
X:\sup_{x \in \mathbb{R}}\| \Phi (x)\|
e^{-\mu | x | }<\infty \right\},
\]%
and the decay norm
\[
\| \Phi \| _{\mu }=\sup_{x \in
\mathbb{R}}\| \Phi (x )\| e^{-\mu |
x | }, \Phi \in B_{\mu }\left( \mathbb{R},\mathbb{R}^{2}\right).
\]%
Then  $\left( B_{\mu }\left( \mathbb{R},\mathbb{R}%
^{2}\right) ,\| \cdot \| _{\mu}\right) $ is a
Banach space.

In this paper, a traveling wave solution of \eqref{1} is a special solution
of the form $X_n(x)=\phi (t), Y_n(x)=\psi(t), t=x+cn$ with the wave
speed $c>0$ and the wave profile $(\phi,\psi)\in X$. Then
$(\phi,\psi)$ and $c$ must satisfy the following recursion system
\begin{equation}\label{2}
\begin{cases}
\phi (t+c)=\int_{\mathbb{R}}\phi (y)e^{r_{1}(1-\phi (y)- \sum_{i=1}^{m}a_{i}\phi (y-ci)-\sum_{i=0}^{m}b_{i}\psi (y-ci))}k_{1}(t-y)dy,t\in\mathbb{R},\\
\psi
(t+c)=\int_{\mathbb{R}} \psi (y)e^{r_{2}(1-\psi (y)-\sum_{i=1}^{m}e_{i}\psi(y-ci)-\sum_{i=0}^{m}f_{i}\phi(y-ci))}k_{2}(t-y)dy,t\in\mathbb{R}.
\end{cases}
\end{equation}

In this paper, similar to those in \cite{lili,libingtuan,linlidc,llrjmb}, we are interested in modeling the simultaneous invasion of two competitors. Therefore, we also require the following asymptotic boundary condition
\begin{equation}\label{a}
\lim_{t\to -\infty}(\phi(t),\psi(t))=(0,0), \, \,\,\, \liminf_{t\to\infty}\phi(t)>0,\,
\liminf_{t\to\infty}\psi(t)>0.
\end{equation}
Clearly, \eqref{2}-\eqref{a} can model the coinvasion-coexistence process of two competitors: at any fixed location $x\in\mathbb{R}$, there was not individual of the both species a long time ago ($n\to -\infty$ such that $t=x+cn\to -\infty$), but two competitors will coexist after a long time ($n\to \infty$ such that $t \to \infty$). In particular, we also investigate the following asymptotic boundary condition
\begin{equation}\label{kkkk}
\lim_{t\to - \infty}(\phi(t),\psi(t))=(0,0),\,\,\,\, \lim_{t\to\infty}\phi(t)=k_1,\,
\lim_{t\to\infty}\psi(t)=k_2,
\end{equation}
in which $k_1>0,k_2>0$ are defined by
\begin{equation*}
\begin{cases}
k_1+\sum_{i=1}^{m}a_{i}k_1+\sum_{i=0}^{m}b_{i}k_2=1, \\
k_2+\sum_{i=1}^{m}e_{i}k_2+\sum_{i=0}^{m}f_{i}k_1=1
\end{cases}
\end{equation*}
provided that
\[
1+ \sum_{i=1}^{m}a_{i}> \sum_{i=0}^{m}f_{i}, 1+ \sum_{i=1}^{m}e_{i}> \sum_{i=0}^{m}b_{i}.
\]
Clearly, the above condition is true if (A5) holds.

We now present some results established by Hsu and Zhao \cite{hsuzhao} and consider the following discrete time recursion
\begin{equation}\label{mon}
\begin{cases}
u_{n+1}(x)=\int_{\mathbb{R}}
b(u_{n}(y))k(x-y)dy,x\in\mathbb{R},n=0,1,2,\cdots, \\
u_0(x)=u(x),x\in\mathbb{R},
\end{cases}
\end{equation}
in which $u(x)$ is bounded and uniformly continuous, $k$ satisfies (A1)-(A2) and $b:\mathbb{R}^+ \to \mathbb{R}^+ $ such that:
\begin{description}
\item[(B1)] there exists $u^+>0$ such that $b(0)=0, b(u^+)=u^+,$  $b(u)>u$ for $ u\in (0, u^+)$ while $b(u)<u$ for $u>u^+;$
    \item[(B2)] for some $U^+\ge u^+,$ $b(u), u\in [0, U^+]$ is continuous and monotone;
        \item[(B3)] $b_0=\lim_{u\to 0+}b(u)/u $ exists  and $b_0>1$ holds such that $b(u)<b_0u, u\in (0,U^+];$
        \item[(B4)] there exists $L>0$ such that
        $b_0u-b(u)<Lu^2, u\in (0, U^+].$
\end{description}

For recursion \eqref{mon}, the following results hold (see \cite{hsuzhao}).
\begin{lemma}\label{com}
Assume that (B1)-(B4) hold and $0\le u(x)\le U^+$ for all $x\in\mathbb{R}.$
\begin{description}
\item[(1)] $0\le u_n(x)\le U^+$ for all $n\in\mathbb{N}, x\in\mathbb{R}.$
    \item[(2)] If $0\le v_n(x)\le U^+, n\in\mathbb{N}\bigcup \{0 \}, x\in\mathbb{R}$ such that
      \[
      v_{n+1}(x)\ge (\le)\int_{\mathbb{R}}
b(v_{n}(y))k(x-y)dy, \,\, v_0(x)\ge (\le)u(x),
      \]
      then $v_n(x)\ge (\le)u_n(x) , n\in\mathbb{N}, x\in\mathbb{R}.$
      \item[(3)]  Define
          \[
          c_0=\inf_{\lambda >0}\frac{\ln (b_0 \int_{\mathbb{R}}e^{\lambda y}k(y)dy)}{\lambda}.
          \]
If $c\in (0,c_0)$ holds and $u(x)>0$ admits nonempty support,  then
\[
\liminf_{n\to\infty}\inf_{|x|<cn}u_n(x)=\limsup_{n\to\infty}\sup_{|x|<cn}u_n(x)=u^+.
\]
\end{description}
\end{lemma}

\section{Existence of Positive Traveling Wave Solutions}
\noindent

In this section, we address (A1)-(A4) and consider the existence of nonnegative solutions of \eqref{2}. Denote
\[
\Delta_i(\lambda,c)=\int_{\mathbb{R}}e^{r_i+\lambda y -\lambda
c}k_i(y)dy,i=1,2,
\]
for $\lambda\in\mathbb{R}$ and $ c\ge 0.$ By (A1)-(A3), $\Delta_i(\lambda,c), i=1,2,$ are well defined and the following result is clear.
\begin{lemma}\label{le2.1}
There exists a positive constant $c^* >0$ such that $c<c^*$ implies that
$\Delta_1(\lambda,c) >1$ or $\Delta_2(\lambda,c) >1$ for any $\lambda \ge 0$
while $c>c^*$  implies that
$\Delta_i(\lambda,c)=1 $ has at least one  positive root
  for  each $i=1,2$.
In
addition, when $c> c^*$ is true, let $\lambda_{i}(c)$ satisfy $\Delta_i(\lambda_i(c),c)=1 $ and $\Delta_i(\lambda,c)>1$ for $ \lambda \in (0,\lambda_i(c)), $ then there exists $\gamma \in (1,2)$ such
that $\Delta_i(\lambda'_{i}(c),c)<1 $ for all $\lambda'_{i}(c) \in
(\lambda_{i}(c), \gamma \lambda_{i}(c)],$ $ i=1,2.$
\end{lemma}

In fact, $\Delta_1(\lambda,c)$ is convex and there exists $c_1^*$ such that $\Delta_1(\lambda,c)=1$ has at least one real root if $c>c^*_1$ while $\Delta_1(\lambda,c)>1$  if $\lambda \ge 0$ and $ c<c^*_1,$ and $c_1$ can also be formulated as follows
\[
c_1^*= \inf_{\lambda >0}\frac{\ln (e^{r_1} \int_{\mathbb{R}}e^{\lambda y}k_1(y)dy)}{\lambda}
\]
by Hsu and Zhao \cite{hsuzhao}, Liang and Zhao \cite{liangzhao}, Weinberger \cite{wein}. In a similar way, we can obtain $c^*_2>0$ by
\[
c_2^*= \inf_{\lambda >0}\frac{\ln (e^{r_2} \int_{\mathbb{R}}e^{\lambda y}k_2(y)dy)}{\lambda}
\]
and $c^*=\max\{c_1^*, c_2^*\}.$

By Lemma \ref{le2.1}, if $c>c^*$ is fixed, then we can define continuous functions
\begin{equation*}
\overline{\phi }(t)=\min \left\{ e^{\lambda _{1}t},~l_1\right\}
, \,\,\,
\overline{\psi }(t)=\min \left\{ e^{\lambda _{2}t},~l_2\right\}
\end{equation*}%
with
\[
l_i=
\begin{cases}
1, \,\,r_i \le 1,\\
\frac{e^{r_i-1}}{r_i},\,\, r_i>1
\end{cases}
\,\,\,\text{ for }i=1,2.
\]
Clearly,
$[0,l_1]\times [0, l_2]$ is an invariant region of  \eqref{ricker}. Further define
\begin{equation*}
\underline{\phi }(t) =
\max\{e^{\lambda _{1}t}-\rho e^{\eta \lambda _{1}t}, 0\},\,\,
\underline{\psi }(t) =
\max\{
e^{\lambda _{2}t}-\rho e^{\eta \lambda _{2}t},0\},
\end{equation*}
where $\rho >1$ is a positive constant clarified later and $\eta \in (1,2)$ is a constant such that
\[
\eta \lambda _{1} < \lambda _{1}+\lambda _{2}, \text{  }\eta
\lambda _{2} < \lambda _{1}+\lambda _{2}, \text{ }\Delta_1(\eta
\lambda _{1},c)<1,\text{ }\Delta_2(\eta \lambda _{2},c)<1.
\]

Using these notations, we give the following potential wave profile set
\[
\Gamma=\{(\phi, \psi)\in X, (\underline{\phi}, \underline{\psi})\le (\phi, \psi)\le (\overline{\phi}, \overline{\psi})\},
\]
which exhibits the following properties.
\begin{lemma}\label{le2.2}
$\Gamma$ is convex and nonempty. Moreover, it is closed and bounded with respect to the decay norm $\|\cdot \|_{\mu}$.
\end{lemma}

Let $P=(P_1,P_2): \Gamma \to X$ be
\begin{eqnarray*}
P_{1}(\phi ,\psi )(t) &=&\int_{\mathbb{R}}\phi (y)e^{r_{1}(1-\phi (y)- \sum_{i=1}^{m}a_{i}\phi (y-ci)-\sum_{i=0}^{m}b_{i}\psi (y-ci))}k_{1}(t-c-y)dy, \\
P_{2}(\phi ,\psi )(t) &=&\int_{\mathbb{R}}\psi (y)e^{r_{2}(1-\psi (y)-\sum_{i=1}^{m}e_{i}\psi(y-ci)-\sum_{i=0}^{m}f_{i}\phi(y-ci))}k_{2}(t-c-y)dy
\end{eqnarray*}
for $(\phi, \psi)\in \Gamma, t\in\mathbb{R}.$ Then  a fixed point of $P$ in $X$ is a solution to \eqref{2}. In what follows, we shall prove the existence of the fixed points of $P$ by Schauder's fixed point theorem.

\begin{lemma}\label{le2.3}
If $\rho >1$ is large, then $P: \Gamma \to \Gamma.$
\end{lemma}
\begin{proof}
For $(\phi, \psi) \in \Gamma,$ it is clear that $P_1$ is nonincreasing in $\psi$ and
\begin{eqnarray*}
P_{1}(\phi ,\psi )(t) &=&\int_{\mathbb{R}}\phi (y)e^{r_{1}(1-\phi (y)- \sum_{i=1}^{m}a_{i}\phi (y-ci)-\sum_{i=0}^{m}b_{i}\psi (y-ci))}k_{1}(t-c-y)dy \\
&\leq &\int_{\mathbb{R}}\phi (y)e^{r_{1}(1-\phi (y))}k_{1}(t-c-y)dy.
\end{eqnarray*}%
Note that $ue^{r_{1}(1-u)}\in \lbrack 0,l_{1}]$ for $u\in \lbrack 0,l_{1}],$ then%
\[
\int_{\mathbb{R}}\phi (y)e^{r_{1}(1-\phi (y))}k_{1}(t-c-y)dy\leq l_{1}
\]%
is clear for $\left( \phi ,\psi \right) $ $\in \Gamma .$ Furthermore, if $%
\left( \phi ,\psi \right) $ $\in \Gamma ,$ then
\begin{eqnarray*}
P_{1}(\phi ,\psi )(t) &\leq &\int_{\mathbb{R}}\phi (y)e^{r_{1}(1-\phi
(y))}k_{1}(t-c-y)dy \\
&\leq &e^{r_{1}}\int_{\mathbb{R}}\phi (y)k_{1}(t-c-y)dy \\
&\leq &e^{r_{1}}\int_{\mathbb{R}}e^{\lambda _{1}y}k_{1}(t-c-y)dy \\
&=&e^{\lambda _{1}t}
\end{eqnarray*}%
by Lemma \ref{le2.1}. Thus, we have proved that
\[
P_{1}(\phi ,\psi )(t)\leq \overline{\phi }(t),\text{ }(\phi, \psi) \in \Gamma,\,\, t\in \mathbb{R}.
\]%
Similarly, we can obtain
\[
P_{2}(\phi ,\psi )(t)\leq \overline{\psi }(t),\text{ }(\phi, \psi) \in \Gamma,\,\,t\in \mathbb{R}.
\]

If $\underline{\phi}(t)=0,$ then it is clear that
\[
P_{1}(\phi ,\psi )(t)\geq 0= \underline{\phi }(t).
\]
We now consider $t<0$ such that $\underline{\phi}(t)>0.$ Clearly, there exists $L>0$ such that
\begin{eqnarray*}
| ue^{r_{1}(1-u-v-w)}-ue^{r_{1}}|  &\leq
&Le^{r_{1}}\left( u^{2}+uv+uw\right) , \\
(u,v,w) &\in &\left[ 0,l_{1}\right] \times \left[ 0,\left(
1+\sum_{i=1}^{m}a_{i}\right) l_{1}\right] \times \left[ 0,%
\sum_{i=0}^{m}b_{i}l_{2}\right] ,
\end{eqnarray*}
and
\begin{eqnarray*}
&&P_{1}(\phi ,\psi )(t) \\
&=&\int_{\mathbb{R}}\phi (y)e^{r_{1}\left( 1-\phi
(y)-\sum_{i=1}^{m}a_{i}\phi (y-ci)-\sum_{i=0}^{m}b_{i}\psi (y-ci)\right)
}k_{1}(t-c-y)dy \\
&\geq &\int_{\mathbb{R}}\phi (y)e^{r_{1}}k_{1}(t-c-y)dy \\
&&-Le^{r_{1}}\int_{\mathbb{R}}\left[ \phi ^{2}(y)+\phi
(y)\sum_{i=1}^{m}a_{i}\phi (y-ci)+\phi (y)\sum_{i=0}^{m}b_{i}\psi (y-ci)%
\right] k_{1}(t-c-y)dy \\
&\geq &\int_{\mathbb{R}}\left( e^{\lambda _{1}y}-\rho e^{\eta \lambda
_{1}y}\right) e^{r_{1}}k_{1}(t-c-y)dy \\
&&-Le^{r_{1}}\int_{\mathbb{R}}\left[ \left( 1+\sum_{i=1}^{m}a_{i}\right)
e^{2\lambda _{1}y}+\sum_{i=0}^{m}b_{i}e^{(\lambda _{1}+\lambda _{2})y}\right]
k_{1}(t-c-y)dy \\
&=&\Delta _{1}(\lambda _{1},c)e^{\lambda _{1}t}-\rho \Delta _{1}(\eta
\lambda _{1},c)e^{\eta \lambda _{1}t} \\
&&-L\left( 1+\sum_{i=1}^{m}a_{i}\right) \Delta _{1}(2\lambda
_{1},c)e^{2\lambda _{1}t}-L\sum_{i=0}^{m}b_{i}\Delta _{1}(\lambda
_{1}+\lambda _{2},c)e^{(\lambda _{1}+\lambda _{2})t} \\
&=&e^{\lambda _{1}t}-\rho \Delta _{1}(\eta \lambda _{1},c)e^{\eta \lambda
_{1}t} \\
&&-L\left( 1+\sum_{i=1}^{m}a_{i}\right) \Delta _{1}(2\lambda
_{1},c)e^{2\lambda _{1}t}-L\sum_{i=0}^{m}b_{i}\Delta _{1}(\lambda
_{1}+\lambda _{2},c)e^{(\lambda _{1}+\lambda _{2})t}.
\end{eqnarray*}
Note that $\rho >1$ and $t<0,$ we see that%
\[
P_{1}(\phi ,\psi )(t)\geq \underline{\phi }(t),  \text{ }(\phi, \psi) \in \Gamma,\,\, t\in \mathbb{R}
\]%
provided that
\[
\rho \geq 1+\frac{L\left( 1+\sum_{i=1}^{m}a_{i}\right) \Delta _{1}(2\lambda
_{1},c)+L\sum_{i=0}^{m}b_{i}\Delta _{1}(\lambda _{1}+\lambda _{2},c)}{%
1-\Delta _{1}(\eta \lambda _{1},c)}.
\]

In a similar way,  if $\rho >1$ is large, then
\[
P_{2}(\phi ,\psi )(t)\geq \underline{\psi }(t),\text{ }(\phi, \psi) \in \Gamma,\,\, t\in \mathbb{R}.
\]

By what we have done, we complete the proof.
\end{proof}

\begin{lemma}\label{le2.4}
Assume that $\mu <\min\{\lambda_1,\lambda_2\}.$ Then $P: \Gamma \to \Gamma$ is complete continuous with respect to the  norm $\| \cdot \| _{\mu}.$
\end{lemma}

The proof is provided by Lin \cite[Lemma 3.4]{lin} and we omit it here.

Applying Schauder's fixed point theorem, we can obtain the following result.
\begin{theorem}\label{th1}
Assume that $c>c^*.$ Then \eqref{2} has a positive solution $(\phi, \psi)$ such that
\[
\phi(t)>0, \,\psi(t)>0, \,t\in\mathbb{R}
\]
and
\[
\lim_{t\to -\infty}(\phi(t),\psi(t))=(0,0).
\]
\end{theorem}
\begin{remark}{\rm
The proof of Theorem \ref{th1} is similar to Lin \cite[Theorem 3.5]{lin}.}
\end{remark}

\section{Asymptotic of Traveling Wave Solutions}
\noindent

In this part, we shall investigate $\lim_{t\to\infty}(\phi(t), \psi(t)),$ in which $(\phi, \psi)$ is given by Theorem \ref{th1} and satisfies $\lim_{t\to -\infty}(\phi(t),\psi(t))=(0,0).$

Firstly, we consider
\begin{equation}\label{01}
\begin{cases}
X_{n+1}=X_{n}e^{r_{1}\left(
1-X_{n}-\sum_{i=1}^{m}a_{i}X_{n-i}-\sum_{i=0}^{m}b_{i}Y_{n-i}\right) },\\
Y_{n+1}=Y_{n}e^{r_{2}\left(
1-Y_{n}-\sum_{i=1}^{m}e_{i}Y_{n-i}-\sum_{i=0}^{m}f_{i}X_{n-i}\right) }.
\end{cases}
\end{equation}
\begin{definition}\label{squ}
{\rm
For $s\in [0,1]$ with
\[
R(s)=(r_1(s),r_2(s))\in \mathbb{R}^2, \,\, T(s)=(t_1(s),t_2(s))\in \mathbb{R}^2,
\]
$[R(s), T(s)]$ is a contracting rectangle of \eqref{01} if
\begin{description}
\item[(C1)] $r_i(s),t_i(s)$ are continuous in $s\in [0,1],i\in \{1,2\};$
\item[(C2)] $r_i(s)$ is strictly increasing in $s$ while $t_i(s)$ is strictly decreasing in $s\in [0,1],i\in\{1,2\}$;
\item[(C3)] $(0,0)\le (r_1(0),r_2(0))< (r_1(1),r_2(1))=(k_1,k_2)= (t_1(1),t_2(1))< (t_1(0),t_2(0))\le (2l_1,2l_2);$
\item[(C4)] for each $s\in (0,1)$
\[
r_1(s)< u_1^0 e^{r_1 (1-u_1^0 -\sum_{l=1}^{l=m}a_lu_1^l-\sum_{l=0}^m b_lu_2^l)}<t_1(s)
\]
and
\[
r_2(s)< u_2^0 e^{r_2 (1-u_2^0 -\sum_{l=1}^{l=m}e_lu_2^l-\sum_{l=0}^m f_lu_1^l)}<t_2(s)
\]
if $u^{l}_{i}\in [r_i(s),t_i(s)], l\in \{0,1,\cdots, m\}$.
\end{description}
}
\end{definition}
The definition was given by Lin \cite{lin}, by which the author obtained the stability of positive steady state of diference systems of higher order.
\begin{lemma}\label{ct}
If (A1)-(A5) hold and $\epsilon >0$ is small, then $[R(s), T(s)]$ is a contracting rectangle of \eqref{01}, where
\[
r_i(s)=sk_i, t_i(s)=sk_i+(1-s)(1+\epsilon),i=1,2.
\]
\end{lemma}

The proof is trivial and we omit it here.

\begin{remark}{\rm
By Lin \cite[Theorem 4.2]{lin}, $(k_1,k_2)$ is asymptotically stable if (A1)-(A5) hold. Moreover, it is evident that $(0,0)$ is unstable if $r_1 >0, r_2 >0$.}
\end{remark}

\begin{lemma}
If (A1)-(A5) hold, then
\[
1\ge \limsup_{t\rightarrow \infty }\phi (t) \ge \liminf_{t\rightarrow \infty }\phi (t)>0
\]
and
\[
1\ge \limsup_{t\rightarrow \infty }\psi (t) \ge \liminf_{t\rightarrow \infty }\psi (t)>0.
\]
\end{lemma}
\begin{proof}
Clearly, $X_n(x)=\phi (t)\in (0,1), t\in\mathbb{R}$ satisfies
\[
\phi (t)\geq \int_{\mathbb{R}}\phi (y)e^{r_{1}(1-\phi
(y)-\sum_{i=1}^{m}a_{i}-\sum_{i=0}^{m}b_{i})}k_{1}(t-c-y)dy
\]%
and%
\[
\begin{cases}
X_{n+1}(x)\geq \int_{\mathbb{R}}X_{n}(y)e^{r_{1}(1-\sum_{i=1}^{m}a_{i}-%
\sum_{i=0}^{m}b_{i}-X_{n}(y))}k_{1}(x-y)dy,x\in\mathbb{R},n=0,1,2,\cdots,\\
X_{0}(x)=\phi (x),x\in\mathbb{R}.
\end{cases}
\]
Consider the following initial value problem
\[
\begin{cases}
Z_{n+1}(x)= \int_{\mathbb{R}}Z_{n}(y)e^{r_{1}(1-\sum_{i=1}^{m}a_{i}-%
\sum_{i=0}^{m}b_{i}-Z_{n}(y))}k_{1}(x-y)dy,x\in\mathbb{R},n=0,1,2,\cdots,\\
Z_{0}(x)=\phi (x),x\in\mathbb{R}.
\end{cases}
\]
It is evident that
\[
ze^{r_{1}(1-\sum_{i=1}^{m}a_{i}-%
\sum_{i=0}^{m}b_{i}-z)}
\]
is monotone increasing in $z\in [0,1].$
By Lemma \ref{com}, we see that
\[
\liminf_{n\to\infty} X_n(0)\ge 1-\sum_{i=1}^{m}a_{i}-%
\sum_{i=0}^{m}b_{i} >0
\]
and
\[
\liminf_{t\to\infty}\phi(t)\ge 1-\sum_{i=1}^{m}a_{i}-%
\sum_{i=0}^{m}b_{i} >0
\]
by the invariant wave profile $\phi(t)$ with $t=x+cn.$

At the same time, we have
\[
\begin{cases}
X_{n+1}(x)\leq \int_{\mathbb{R}%
}X_{n}(y)e^{r_{1}(1-X_{n}(y))}k_{1}(x-y)dy,x\in\mathbb{R},n=0,1,2,\cdots,\\
X_{0}(x)=\phi (x),x\in\mathbb{R}.
\end{cases}
\]
Then Lemma \ref{com} indicates that
\[
\limsup_{n\to\infty} X_n(0) \le 1
\]
and
\[
\limsup_{t\to\infty}\phi(t)\le 1.
\]

In a similar way, we can prove that
\[
1\ge \limsup_{t\rightarrow \infty }\psi (t) \ge \liminf_{t\rightarrow \infty }\psi (t)>0.
\]

The proof is complete.
\end{proof}
\begin{theorem}\label{th2}
Assume that (A1)-(A5) hold and $(\phi,\psi)$ is formulated by Theorem \ref{th1}. Then \eqref{kkkk} is true.
\end{theorem}
\begin{proof}
By what we have done, there exists $s_0\in (0,1)$ such that
\[
t_{1}(s_{0})>\limsup_{t\rightarrow \infty }\phi (t)\geq
\liminf_{t\rightarrow \infty }\phi (t)>r_{1}(s_{0})>0
\]%
and
\[
t_{2}(s_{0})>\limsup_{t\rightarrow \infty }\psi (t)\geq
\liminf_{t\rightarrow \infty }\psi (t)>r_{1}(s_{0})>0,
\]
where $t_1(s),t_2(s),r_1(s),r_2(s)$ are defined by Lemma \ref{ct}. From Lin \cite[Theorem 4.3]{lin}, we complete the proof.
\end{proof}

In the above Theorem \ref{th2}, we have proved the asymptotic behavior \eqref{kkkk} when $r_1,r_2\in (0,1].$ We now investigate
the asymptotic boundary condition \eqref{a} for $r_1,r_2\in (0,\infty)$.
\begin{theorem}\label{th3}
Assume that (A1)-(A4) and
\begin{equation}\label{up}
\sum_{i=1}^{m}a_{i}l_1+\sum_{i=0}^{m}b_{i}l_2<1, \sum_{i=1}^{m}e_{i}l_2+\sum_{i=0}^{m}f_{i}l_1<1.
\end{equation}
If $(\phi,\psi)$ is formulated by Theorem \ref{th1}, then \eqref{a} is true.
\end{theorem}
\begin{proof}
Clearly,
$X_n(x)=\phi (t)\in (0,1), t\in\mathbb{R}$ satisfies
\[
\phi (t)\geq \int_{\mathbb{R}}\phi (y)e^{r_{1}(1-\phi
(y)-\sum_{i=1}^{m}a_{i}l_1-\sum_{i=0}^{m}b_{i}l_2)}k_{1}(t-c-y)dy
\]%
and
\[
\begin{cases}
X_{n+1}(x)\geq \int_{\mathbb{R}}X_{n}(y)e^{r_{1}(1-\sum_{i=1}^{m}a_{i}l_1-%
\sum_{i=0}^{m}b_{i}l_2-X_{n}(y))}k_{1}(x-y)dy,x\in\mathbb{R},n=0,1,2,\cdots,\\
X_{0}(x)=\phi (x),x\in\mathbb{R}.
\end{cases}
\]
Let $b(u)$ be defined by
\[
b(u)=\inf_{x\in [u,l_1+1]} x e^{r_{1}(1-\sum_{i=1}^{m}a_{i}l_1-%
\sum_{i=0}^{m}b_{i}l_2-x)} , u\in [0,l_1+1].
\]
Then $b(u)$ satisfies (B1)-(B4) with $U^+=l_1+1$ and some $u^+\in (0, l_1].$ From Lemma \ref{com}, we further obtain
\[
\liminf_{n\to\infty} X_{n}(0)\ge u^+ >0,
\]
and
\[
\liminf_{t\to\infty}\phi(t)\ge u^+ >0.
\]

By a similar discussion on $\psi (t),$ we complete the proof of \eqref{a}.
\end{proof}
\begin{remark}{\rm
In Theorem \ref{th2}, we obtain \eqref{a} for any $r_1,r_2 \in (0,\infty)$ with \eqref{up} even if the comparison principle fails. Although \eqref{kkkk} maybe fails, any positive solutions satisfying \eqref{2} with \eqref{a} still describe the successful invasion of two competitors.
}
\end{remark}

\section{Nonexistence of Traveling Wave Solutions}
\noindent

In this section, we shall consider the nonexistence of $(\phi,\psi)$ satisfying
\begin{equation}\label{asy}
\lim_{t\to - \infty}(\phi(t),\psi(t))=(0,0), \,\,\,  \liminf_{t\to\infty}\phi(t)>0,\,
\liminf_{t\to\infty}\psi(t)>0.
\end{equation}

\begin{lemma}\label{le4.1}
Assume that $(\phi,\psi)$ satisfying \eqref{asy} is a positive solution to \eqref{2}. Then
\[
0<\phi(t)\le \frac{e^{r_1-1}}{r_1},\,\,\,  0<\psi(t)\le \frac{e^{r_2-1}}{r_2},\,\,\, t\in\mathbb{R}.
\]
\end{lemma}
\begin{proof}
Because of $k_1$ is Lebesgue integrable, then there exists a nonempty interval $[a,b]\subseteq \mathbb{R}$ such that $k_1(y)>0, a.e. y\in [a,b]$ with $b> a\ge 0.$ If $\phi(t_1)=0,$ then the continuity of $\phi(t)$ implies that
\[
\phi(t)=0, t\in [t_1+a,t_1+b], t\in [t_1-b, t_1-a].
\]
If $a=0,$ then $\phi(t)=0,t\in [t_1-2b,t_1+2b]$ by replacing $t_1$ by $t_1\pm b.$ By mathematical induction, we see that $\phi(t)=0, t\in \mathbb{R}.$

If $a\ne 0,$ then we replace $t_1$ by $t_1\pm \frac{a+b}{2},$ then we see that
\[
\phi(t)=0, t\in [t_1- (b-a), t_1+ (b-a)].
\]
After selecting finite points, we can prove that
\[
\phi(t)=0, t\in [t_1- b, t_1+b]
\]
and further have
\[
\phi(t)=0, t\in [t_1-n b, t_1+nb], n\in \mathbb{N},
\]
which implies that $\phi(t)=0,t\in\mathbb{R}.$ In a similar way, we can verify that
$
\psi(t)=0,t\in\mathbb{R}
$
if $\psi(t_2)=0$ for some $t_2\in\mathbb{R}.$

Moreover, since
\[
xe^{r_{1}(1-x-a_{1}y)} \le \frac{e^{r_1-1}}{r_1},\text{ \ \ \ }ye^{r_{2}(1-y-a_{2}x)}\le \frac{e^{r_2-1}}{r_2}
\]
for $x > 0, y > 0,$ then
\[
0<\phi(t)\le \frac{e^{r_1-1}}{r_1},\,\,\,  0<\psi(t)\le\frac{e^{r_2-1}}{r_2}.
\]
The proof is complete.
\end{proof}

We now present our main result of this section.
\begin{theorem}\label{th5}
Assume that $c<c^*.$ Then \eqref{2} has no positive solutions satisfying \eqref{asy}.
\end{theorem}
\begin{proof}
We assume that $\Delta_1(\lambda, c)>1$ for any $\lambda\in\mathbb{R}, c\in (0,c^*).$ Were the statement false, there exists $c_1\in (0, c^*)$ such that \eqref{2}  has a solution $(\phi, \psi)$ satisfying \eqref{asy}, then
\[
0<\phi(t)\le \frac{e^{r_1-1}}{r_1},\,\,\,  0<\psi(t)\le \frac{e^{r_2-1}}{r_2},\,\,\, t\in\mathbb{R}
\]
by Lemma \ref{le4.1}.
Let $2c'=c_1+c^*$ and $\epsilon \in (0,1)$ such that
\[
\Delta (\lambda, c)=\int_{\mathbb{R}}e^{r_1(1-\epsilon)+\lambda y -\lambda
c}k_1(y)dy >1 \text{ for any } \lambda\in\mathbb{R}, c\in (0,c').
\]

By the asymptotic of $\phi,\psi,$ we can choose $T<0$ such that
\[
\phi (y)+ \sum_{i=1}^{m}a_{i}\phi (y-ci)+\sum_{i=0}^{m}b_{i}\psi (y-ci) < \epsilon, \,\, y<T.
\]
If $t>T,$ then \eqref{asy}  implies that there exists $M>1$ such that
\[
\phi (y)+ \sum_{i=1}^{m}a_{i}\phi (y-ci)+\sum_{i=0}^{m}b_{i}\psi (y-ci)<M\phi(y), \, y\ge T.
\]

Therefore, we obtain
\[
\phi (t+c_1)\ge \int_{\mathbb{R}}\phi (y)e^{r_{1}(1-\epsilon-M\phi (y))}k_{1}(t-y)dy, t\in\mathbb{R}.
\]

Let $\phi(t)=X'_n(x),t=x+cn, $ then
\[
\begin{cases}
X'_{n+1}(x)\ge \int_{\mathbb{R}}X'_n(y)e^{r_{1}(1-\epsilon-MX'_n(y))}k_{1}(x-y)dy, x\in\mathbb{R},n=0,1,2,\cdots, \\
X'_0(x)=\phi(x) >0, x\in\mathbb{R}.
\end{cases}
\]
Define a continuous function $\underline{b}(u)$ by
\[
\underline{b}(u)=\min_{v\in \left[u, \frac{e^{r_1-1}}{r_1}\right]}ve^{r_1(1-\epsilon -Mv)},
\]
then $\underline{b}(u)=u$ has a unique root $u_*\in (0, \frac{e^{r_1-1}}{r_1}].$
We further have
\[
\begin{cases}
X'_{n+1}(x)\ge \int_{\mathbb{R}}\underline{b}(X'_n(y))k_{1}(x-y)dy, x\in\mathbb{R},n=0,1,2,\cdots, \\
X'_0(x)=\phi(x) >0, x\in\mathbb{R}.
\end{cases}
\]
in which $X'_n(x)\in \left[0, \frac{e^{r_1-1}}{r_1}\right]$ and $\underline{b}(u)$ is monotone increasing for $u\in \left[0, \frac{e^{r_1-1}}{r_1}\right].$ By Lemma \ref{com}, the following result holds
\begin{equation}\label{asym}
\liminf_{n\to\infty}\inf_{|x|<cn} X'_n(x)\ge u_* >0, c\in (0,c').
\end{equation}

Letting $-2x=(c_1+c')n$ and $n\to\infty,$ then $x+c_1n\to -\infty$ holds and a contradiction occurs between \eqref{asy} and \eqref{asym}. The proof is complete.
\end{proof}


\end{document}